\documentclass[12pt]{article}

\usepackage{amssymb}
\usepackage{amscd}
\usepackage{amsthm}
\usepackage{amsmath}
\usepackage{latexsym}
\usepackage{graphicx}
\usepackage{delarray}
\usepackage{float}
\usepackage[dvips]{epsfig}
\usepackage{subfig}
\usepackage[usenames]{color}

\usepackage{hyperref}\usepackage{breakurl}
\DeclareMathOperator{\esssup}{ess\,sup}
\DeclareMathOperator{\essinf}{ess\,inf}
\DeclareMathOperator{\supp}{supp}

\begin{document}

\theoremstyle{plain}
\newtheorem{theorem}{Theorem}[section]
\newtheorem{claim}{Claim}
\newtheorem{lemma}[theorem]{Lemma}
\newtheorem{proposition}[theorem]{Proposition}
\newtheorem{corollary}[theorem]{Corollary}

\theoremstyle{remark}
\newtheorem{example}[theorem]{Example} 	
\newtheorem{remark}{Remark}
\theoremstyle{definition}
\newtheorem{definition}{Definition}
\hfuzz5pt 


\newcommand{\gt}{\tilde{g}}
\newcommand{\R}{\mathbb{R}}
\newcommand{\Z}{\mathbb{Z}}
\newcommand{\N}{\mathbb{N}}
\newcommand{\Zt}{\mathbb{Z}^2}
\newcommand{\Zd}{\mathbb{Z}^d}
\newcommand{\Ztd}{0\mathbb{Z}^{2d}}
\newcommand{\Rt}{\R^2}
\newcommand{\Rtd}{\R^{2d}}
\newcommand{\Zp}{Z_p}
\newcommand{\al}{\alpha}
\newcommand{\be}{\beta}
\newcommand{\om}{\omega}
\newcommand{\Om}{\Omega}
\newcommand{\ga}{\gamma}
\newcommand{\bga}{\boldsymbol\ga}
\newcommand{\bg}{\mathbf{g}}
\newcommand{\de}{\delta}
\newcommand{\la}{\lambda}
\newcommand{\La}{\Lambda}
\newcommand{\ala}{\la^\circ}
\newcommand{\aLa}{\La^\circ}
\newcommand{\nat}{\natural}
\newcommand{\G}{\mathcal{G}}
\newcommand{\A}{\mathcal{A}}
\newcommand{\V}{\mathcal{V}}
\newcommand{\M}{\mathcal{M}}
\newcommand{\MV}{\mathcal{MV}}
\newcommand{\MP}{\mathcal{MP}}
\newcommand{\Sp}{\mathcal{S}}
\newcommand{\W}{\mathcal{W}}
\newcommand{\Hp}{\mathcal{H}}
\newcommand{\ka}{\kappa}
\newcommand{\cast}{\circledast}
\newcommand{\id}{\mbox{Id}}
\newcommand{\by}{\mathbf{Y}}
\newcommand{\bx}{\mathbf{X}}
\newcommand{\bz}{\mathbf{Z}}
\newcommand{\bn}{\mathbf{N}}
\newcommand{\bv}{\mathbf{v}}
\newcommand{\bu}{\mathbf{u}}
\newcommand{\bA}{\mathbf{A}}
\newcommand{\bC}{\mathbf{C}}
\newcommand{\bD}{\mathbf{D}}
\newcommand{\bh}{\mathbf{h}}
\newcommand{\bff}{\mathbf{f}}
\newcommand{\bH}{\mathbf{H}}
\newcommand{\bV}{\mathbf{V}}
\newcommand{\bX}{\mathbf{x}}
\newcommand{\byy}{\mathbf{y}}
\newcommand{\bZ}{\mathbf{z}}
\newcommand{\bt}{\mathbf{t}}
\newcommand{\bs}{\mathbf{\sigma}}
\newcommand{\bI}{\mathbf{I}}

\newcommand{\lo}{{\ell^1}}
\newcommand{\Lo}{{L^1}}
\newcommand{\Lt}{{L^2}}
\newcommand{\SO}{{S_0}}
\newcommand{\SOc}{{S_{0,c}}}
\newcommand{\Lp}{{L^p}}
\newcommand{\Los}{{L^1_s}}
\newcommand{\WCl}{{W(C_0,\ell^1)}}
\newcommand{\lt}{{\ell_2}}

\newcommand{\conv}[2]{{#1}\,\ast\,{#2}}
\newcommand{\twc}[2]{{#1}\,\nat\,{#2}}
\newcommand{\mconv}[2]{{#1}\,\cast\,{#2}}
\newcommand{\set}[2]{\Big\{ \, #1 \, \Big| \, #2 \, \Big\}}
\newcommand{\inner}[2]{\langle #1,#2\rangle}
\newcommand{\innerBig}[2]{\Big \langle #1,#2 \Big \rangle}
\newcommand{\dotp}[2]{ #1 \, \cdot \, #2}

\newcommand{\Zpd}{\Zp^d}
\newcommand{\I}{\mathcal{I}}
\newcommand{\J}{\mathcal{J}}
\newcommand{\Zq}{Z_q}
\newcommand{\Zqd}{Zq^d}
\newcommand{\Zak}{\mathcal{Z}_{a}}
\newcommand{\C}{\mathbb{C}}
\newcommand{\F}{\mathcal{F}}

\newcommand{\convL}[2]{{#1}\,\ast_L\,{#2}}
\newcommand{\abs}[1]{\lvert#1\rvert}
\newcommand{\absbig}[1]{\big\lvert#1\big\rvert}
\newcommand{\absBig}[1]{\Big\lvert#1\Big\rvert}
\newcommand{\scp}[1]{\langle#1\rangle}
\newcommand{\norm}[1]{\lVert#1\rVert}
\newcommand{\normmix}[1]{\lVert#1\rVert_{\ell^{2,1}}}
\newcommand{\normbig}[1]{\big\lVert#1\big\rVert}
\newcommand{\normBig}[1]{\Big\lVert#1\Big\rVert}


\title{Nonstationary Gabor Frames - Approximately Dual  Frames and Reconstruction Errors}

\author{Monika D\"orfler, Ewa Matusiak
\thanks{This work was supported by the WWTF project {\em Audiominer} (MA09-24)}
\thanks{The authors are with the Department of Mathematics,
NuHAG, University of Vienna, Nordbergstrasse 15, 1090 Wien, Austria (e-mail:
monika.doerfler@univie.ac.at, ewa.matusiak@univie.ac.at)}
}

\maketitle


\begin{abstract}
Nonstationary Gabor frames, recently introduced in adaptive signal analysis,  represent a natural generalization of classical Gabor frames by allowing for adaptivity of windows and lattice in either time or frequency. Due to the lack of a complete lattice structure, perfect reconstruction is in general not feasible from coefficients obtained from nonstationary Gabor frames. In this paper it is shown that for nonstationary Gabor frames that are related to some known frames for which dual frames can be computed, good approximate reconstruction can be achieved by resorting to approximately dual frames. In particular, we give constructive examples for so-called almost painless nonstationary frames, that is, frames that are closely related to nonstationary frames with compactly supported windows.   The  theoretical  results are illustrated by concrete  computational and numerical examples.\\


\end{abstract}



\section{Introduction}\label{sec:intro}
Adapted and adaptive signal representation have received increasing interest over the past few years. As opposed to classical approaches such as the short-time Fourier transform (STFT) or wavelet transform, adaptive representations allow for a variation of parameters such as window width or sampling density over time, frequency or both. Changing parameters in the frequency domain leads, for example, to non-uniform filter banks while adapting window width and sampling density in time is reminiscent of the approach suggested in the construction of nonuniform lapped transforms.  
Transforms featuring simultaneous adaptivity in time and frequency are notoriously difficult to construct and implement, cp.~\cite{ro11,do11,jato07}; however they have shown to be useful in some applications , cf.~\cite{baliro11}.
On the other hand, fast and efficient implementations exist for representations with adaptivity in only time or frequency. One recent method to obtain this kind of representations is represented by \emph{nonstationary Gabor frames}, first suggested in \cite{ja06} and further developed in~\cite{badohojave11,dogrhove11,dogrhove12}. All the known implementations rely on compactness of the used analysis window in either time of frequency. This assumption allows for usage of tools developed for painless non-orthogonal expansions~\cite{dagrme86}. While a priori very convenient, the restriction to using compactly supported windows in the domain for which one wishes a flexible representation can be undesirable. As an example, we mention the construction of nonuniform filter banks via nonstationary Gabor frames, in which case this restriction forbids finite impulse response (FIR) filters; the latter are, however, imperative for real-time processing applications. 

In the current contribution, we therefore go beyond the results presented in the references above and consider nonstationary Gabor frames with fast decay but unbounded support. The existence of this kind of frames was shown in~\cite{doma12}. Here  we are concerned with methods for approximate reconstruction for these adaptive systems.

\section{Notation and Preliminaries}

Given a non-zero function $g\in\Lt(\R)$, a modulation, or frequency shift,
operator $M_{bl}$ is defined by $M_{bl}g(t):= e^{2\pi i bl t} g(t)$, and
time shift operator $T_{ak}$ by $T_{ak}g(t):= g(t-ak)$. A composition, $g_{k,l}=M_{bl}
T_{ak} g(t):= e^{2\pi i bl t} g(t-ak)$ is a time-frequency shift operator.

The set $\mathcal{G}(g,a,b) = \{g_{k,l}\,:\, k,l\in\Z \}$ is called a Gabor
system for any real, positive $a,b$. $\mathcal{G}(g,a,b)$ is a Gabor frame for
$\Lt(\R)$, if there exist frame bounds $0<A \leq B < \infty$ such that for every
$f\in\Lt(\R)$ we have
\begin{equation}\label{Eq:framecond}
A \norm{f}_2^2 \leq \sum_{k,l\in\Z} \abs{\inner{f}{g_{k,l}}}^2 \leq B \norm{f}_2^2\,.
\end{equation}

When working with irregular grids,  we assume that the sampling points form a
separated set: a set of sampling points $\{a_k\,:\, k\in\mathbb{Z}\}$ is called
$\delta$-separated, if  $\abs{a_k-a_m} > \delta$ for $a_k$, $a_m$, whenever
$k\neq m$. $\chi_{I}$ will denote the characteristic function of the interval
$I$.

A convenient class of window functions for time-frequency analysis on $\Lt(\R)$
is the Wiener space.
\begin{definition}
A function $g\in L^{\infty}(\R)$ belongs to the Wiener space
$W(L^{\infty},\ell^1)$ if
\begin{equation*}
\norm{g}_{W(L^{\infty},\ell^1)} := \sum_{k\in\Z}
\mbox{ess sup}_{t\in Q} \abs{g(t+k)} < \infty\,,\quad Q=[0,1]\,.
\end{equation*}
\end{definition}

For $g\in W(L^{\infty},\ell^1)$ and $\delta > 0$ we have \cite{gr01}
\begin{equation}\label{eq:wiener_norm}
\esssup_{t\in\R} \sum_{k\in\Z} \abs{g(t-\delta k)} \leq (1+\delta^{-1})
\norm{g}_{W(L^{\infty},\ell^1)}\,.
\end{equation}

For $f\in\Lt(\R)$ we use the following normalization of the  Fourier transform, which we denote by $\F$:
\begin{equation*}
\F f(\om) = \widehat{f}(\om) = \int_{\R} f(t) e^{-2\pi i \om t} \, dt\,. 
\end{equation*}
\section{Nonstationary Gabor Frames}

Nonstationary Gabor systems are a generalization of classical Gabor systems
of regular time-frequency shifts of a single window function. 
\begin{definition}
Let  $\mathbf{g}=\{g_k\in W(L^{\infty},\ell^1): \;k\in\Z\}$ be  a set of window
functions and let $\mathbf{b} = \{b_k: \;k\in\Z\}$  be a corresponding sequence
of frequency-shift parameters. Set $g_{k,l}  = M_{b_k l} g_k$. Then, the set 
\begin{equation}\label{eq:Gabor_system}
\mathcal{G}(\mathbf{g},\mathbf{b})=\{g_{k,l}:\; k,l\in\mathbb{Z}\}
\end{equation}
is called a {\it nonstationary Gabor (NSG) system}.
\end{definition}

In generalization of  regular Gabor frames, for which $g_k = T_{ak}g$, we will usually assume that the windows $g_k$ are localized around  points  $a_k$ in a  separated set of time-sampling points $\{ a_k\, : \, k\in\Z \}$. Further, we will always make the assumption that the frequency sampling parameters $b_k$ are positive numbers  contained in a closed interval, i.e.
 $b_k \in [b_L,b_U]\subset\mathbb{R}^+$ for all $k\in\Z$.

To every collection \eqref{eq:Gabor_system} we associate the analysis operator
$C_g$ given by $(C_g f)_{k,l} = \inner{f}{g_{k,l}}$, and synthesis operator
$U_g$, where $U_g c = \sum_{k,l\in\Z} c_{k,l} g_{k,l}$ and $c\in\ell^2$. For two
Gabor systems $\mathcal{G}(\mathbf{g},\mathbf{b})$ and
$\mathcal{G}(\boldsymbol\ga,\mathbf{b})$ the composition $S_{g,\ga}=U_\ga C_g$, 
\begin{equation}\label{Eq:FOnon}
S_{g,\ga}f = \sum_{k,l\in\Z} \inner{f}{M_{lb_k}g_k} M_{lb_k}\ga_k\,,
\end{equation}
admits a Walnut representation for all $f\in\Lt(\R)$,~\cite{doma12}:
\begin{equation}\label{eq:walnut}
S_{g,\ga}f(t) = \sum_{k,l\in\Z} b_k^{-1} \overline{g_k(t-lb_k^{-1})} \ga_k(t)
f(t-lb_k^{-1})\,.
\end{equation}
 We will frequently use the following correlation functions of a pair of Gabor systems:
\begin{equation}
	G^{g,\gamma}_l(t) = \sum_{k\in\Z}
b_k^{-1} \abs{g_k(t-lb_k^{-1})}\abs{\ga_k(t)}, \mbox{ for } l\in\Z .
\end{equation}
Note that this definition is asymmetric with respect to $\bf{g}$ and $\boldsymbol\gamma$. Using this notation, we obtain the following bounds for the frame operator \eqref{Eq:FOnon}:
\begin{equation}\label{eq:norm_1}
\norm{S_{g,\ga}}^2 \leq \esssup \sum_{l\in\Z}
	G^{g,\gamma}_l \,\cdot \, \esssup \sum_{l\in\Z}
	G^{\gamma,g}_l \,.
\end{equation}
By inspection of \eqref{eq:walnut}, we note that the summands corresponding to $l\neq 0$ may be seen as the off-diagonal entries of the frame operator. We thus isolate the diagonal part 
\begin{equation}\label{eq:diag}
	G^{g,\gamma}_0 (t )  = \sum_{k\in\Z} b_k^{-1} |g_k(t)|
|\ga_k  (t)|
\end{equation}
and denote the off-diagonal entries as follows:
\begin{equation}\label{eq:R}
R_{g,\ga} = \esssup \sum_{l\in\Z\setminus \{0\}}
\sum_{k\in\Z} b_k^{-1} \abs{g_k(\cdot-lb_k^{-1})}\abs{\ga_k(\cdot)}\,.
\end{equation}
Note that, if $\bg = \bga$, then the diagonal part of the frame operator $S_{g,g}$ is equal to  $G^{g,g}_0$.\\
Using this notation, we  obtain the  following additional bound:
\begin{equation}\label{eq:upper_frame_bound}
\inner{S_{g,\ga} f}{f} \, \norm{f}_2^{-2} \leq  \esssup
	G^{g,\gamma}_0 + \sqrt{ R_{g,\ga} \cdot R_{\ga,g} }\,,
\end{equation}

Bessel sequences are of particular importance in the theory of frames and Riesz bases,\cite{chsh94,yo01}. In the regular Gabor case, where $g_k(t)=g(t-ak)$ for some $a>0$, it is sufficient to assume $g\in W(L^{\infty},\ell^1)$ to obtain a Bessel sequence. We next provide a generalization of this property to NSG frames.

\begin{proposition}\label{prop:Bessel condition}
Let $\mathcal{G}(\mathbf{g},\mathbf{b})$ be a NSG system. If
$g_k\in W(L^{\infty},\ell^1)$ for all $k\in\Z$ with $\sup_{k\in\Z} 
\norm{g_k}_{W(L^{\infty},\ell^1)}$ bounded, and $\sum_{k\in\Z}
\abs{g_k(t)}\leq B$ almost everywhere for some $B<\infty$, then the sequence
$g_{k,l}$ is a Bessel sequence.
\end{proposition}

\begin{proof}
Let $f\in \Lt(\R)$. Then by the assumption on the windows $g_k$ and estimate \eqref{eq:norm_1}
\begin{align*}
\sum_{k,l\in\Z} \abs{\inner{f}{g_{k,l}}}^2 &= \inner{S_{g,g}f}{f} \leq
\norm{f}_2^2\cdot  \esssup \sum_{l\in\Z}
	G^{g,g}_l\\
& = \norm{f}_2^2\cdot \esssup \sum_{k\in\Z} b_k^{-1}|g_k(\cdot)|\sum_{l\in\Z} |g_k (\cdot-l b_k^{-1} )|\\
& \leq \norm{f}_2^2\cdot \esssup \sum_{k\in\Z} b_k^{-1}|g_k(\cdot)|(1+b_k^{-1})\|g_k\|_{W(L^{\infty},\ell^1)}\\
 &\leq \norm{f}_2^2\cdot B\cdot  \sup_{k\in\Z} [(1+b_k)\|g_k\|_{W(L^{\infty},\ell^1)}].
\end{align*}
\end{proof}

Given a frame, it is well known that there exists at least one dual
frame $\mathcal{G}(\boldsymbol\gamma,\mathbf{b})$ such that 
\begin{equation}\label{eq:recon}
f = \sum_{k,l\in\Z} \inner{f}{\ga_{k,l}}g_{k,l}\,,\quad \mbox{for all  }f\in\Lt(\R)\,.
\end{equation}
The canonical dual frame is given by $\ga_{k,l} =S^{-1} g_{k,l}$. In the regular Gabor case, the dual frames are again Gabor frames, i.e., they consist of time-frequency shifted versions of one dual window. This is due to the fact, that the frame operator $S$ commutes with  time-frequency shifts, hence
$\ga_{k,l} =S^{-1} g_{k,l} =  S^{-1} M_{b l}T_{ak}g =  M_{b l}T_{ak}S^{-1}g = M_{b l}T_{ak}\ga$. In general, we cannot expect, that the dual frame of a NSG frame is again a NSG frame.\\
However, even in
the case of regular Gabor frames, it is often difficult to calculate a dual frame explicitly. For that reason, alternative possibilities to obtain perfect or approximate reconstruction have been proposed,~\cite{chla10,bafehakr06-1}. The following lemma
quantifies the reconstruction error using 
general pairs of Bessel sequences. 

\begin{lemma}\label{lem:duality_condition}
Let $\mathcal{G}(\mathbf{g},\mathbf{b})$ and
$\mathcal{G}(\boldsymbol\gamma,\mathbf{b})$ be two Bessel sequences. Then 
\begin{equation}\label{eq:duality_estimate}
\norm{I-S_{g,\ga}} \leq \normBig{1 -
	 \sum_{k\in\Z}
b_k^{-1} \overline{g_k}\ga_k}_{\infty} +  \sqrt{ R_{g,\ga} \cdot R_{\ga,g} }\,.
\end{equation}
\end{lemma}

\begin{proof}
Starting from  the Walnut representation of $S_{g,\ga}$, we estimate
 using Cauchy-Schwartz inequality for sums and integrals and, since all summands have absolute value, Fubini's theorem to justify  changing  the order of summation and integral:

\begin{align}\label{eq:lemma}
|\inner{f-S_{g,\ga}f}{f}| & = \Big |\inner{f - \sum_{k\in\Z} b_k^{-1} \overline{g_k}
\ga_k f}{f} - \innerBig{\sum_{l\in\Z\setminus \{0\}} \sum_{k\in\Z} b_k^{-1}
\ga_k(\cdot)\overline{g_k(\cdot-lb_k^{-1}}) f(\cdot - lb_k^{-1})}{f}\Big|\nonumber \\
& \leq \normBig{1-\sum_{k\in\Z}
b_k^{-1} \overline{g_k}\ga_k
}_{\infty}
\norm{f}_2^2 +\sqrt{ R_{g,\ga} \cdot R_{\ga,g} } \,\norm{f}_2^2, 
\end{align}
since
\begin{align}\label{eq:estimate_R}
&\Big |\innerBig{\sum_{\substack{k\in\Z\\l\in\Z\setminus \{0\}} } b_k^{-1}
\ga_k(\cdot)\overline{g_k(\cdot-lb_k^{-1}}) f(\cdot - lb_k^{-1})}{f}\Big |\\\nonumber
&\leq \sum_{\substack{k\in\Z\\l\in\Z\setminus \{0\}} } b_k^{-1} \int_{\R}\abs{\ga_k(t)}
\abs{g_k(t-lb_k^{-1})} \abs{f(t-lb_k^{-1})}
\abs{f(t)} \, dt \nonumber \\
& \leq\sum_{\substack{k\in\Z\\l\in\Z\setminus \{0\}} } b_k^{-1} \left [
\int_{\R} \abs{g_k(t-lb_k^{-1})} \abs{\ga_k(t)}\abs{f(t-lb_k^{-1})}^2
\, dt \right ]^{1/2} \left [ \int_{\R} \abs{g_k(t-lb_k^{-1})}
\abs{\ga_k(t)}\abs{f(t)}^2\, dt \right]^{1/2} \nonumber \\\nonumber
& \leq \left [\int_{\R} \abs{f(t)}^2 \sum_{\substack{k\in\Z\\l\in\Z\setminus \{0\}}}  b_k^{-1}\abs{g_k(t)}
\abs{\ga_k(t-lb_k^{-1})} \, dt \right ]^{1/2} \left [ \int_{\R} \abs{f(t)}^2
\sum_{\substack{k\in\Z\\l\in\Z\setminus \{0\}} }  b_k^{-1}\abs{\ga_k(t)} \abs{g_k(t-lb_k^{-1})} \, dt \right
]^{1/2}.\end{align}
\end{proof}
A special class of NSG systems are
collections of compactly supported windows. They were first addressed in
\cite{badohojave11}. The collection $\mathcal{G}(\mathbf{g},\mathbf{b})$ with
windows $g_k$ being compactly supported with $|\supp g_k| \leq\frac{1}{b_k}$ for all $k$ is a frame for $\Lt(\R)$ if there exist
constants $A>0$ and $B<\infty$ such that
\begin{equation}\label{eq:painless_frame}
A \leq G_0^{g,g} (t) \leq B\, \mbox{ a.e.}\,.
\end{equation}
In this situation,  $\mathcal{G}(\mathbf{g},\mathbf{b})$ is called {\it painless
NSG frame}. The canonical dual atoms are given by  $\ga_{k,l} = M_{lb_k}
 (G_0^{g,g} g)^{-1} g_k$.
Note again that,  in general, we may have $\ga_{k,l} = S^{-1} (M_{lb_k} g_k)\neq   M_{lb_k} (S^{-1}g_k)$.  If $b_k = b$ for all $k$, then the frame operator commutes with the frequency-shifts and the dual frame is an NSG frame.
 
 The existence of  NSG frames with  not necessarily compactly
supported windows was established in \cite{doma12}. For these frames,  finding
canonical dual frames requires the inversion of the  frame operator. This computation is  expensive since the operator has considerably less structure than the frame operator in the classical, regular Gabor frame case, for which fast algorithms now exist,~\cite{st01-1,jaso07,so12}. To circumvent the
problem, we suggest the use of windows other than canonical duals to obtain sufficiently good approximate reconstruction.


\section{Approximately dual atoms}

The notion of approximately dual pairs was discussed in \cite{chla10}. For NSG Bessel sequences we adapt their definition  as follows.

\begin{definition}\label{Def:AppDual}
Two Bessel sequences $\mathcal{G}(\mathbf{g},\mathbf{b})$ and
$\mathcal{G}(\boldsymbol\gamma,\mathbf{b})$ are said to be
approximately dual frames if $\norm{I-S_{g,\ga}} <1$ or $\norm{I-S_{\ga,g}} <1$.
\end{definition}
Note that the two conditions given in the definition are equivalent since
\[\norm{I-S_{g,\ga}}  = \norm{I-C_g U_\ga }  = \norm{I-U_\ga^* C_g^*}= \norm{I-C_\ga U_g}=\norm{I-S_{\ga,g}}.\]
In Definition~\ref{Def:AppDual} it is implicitly stated that, if  two Bessel sequences are approximately dual frames, then each of them is a frame. This result was proved in~\cite{chla10} for general frames and it will be useful to reformulate the conditions for the existence of NSG frame, given in \cite{doma12}, in the context of approximately dual frames.

\begin{proposition}\label{prop:approx_dual}
Let $\mathcal{G}(\mathbf{g},\mathbf{b})$ be a Bessel sequence with Bessel bound $B$ and 
$0<A_1\leq\sum_{k\in\Z} \abs{g_k(t)}^2 \leq A_2 <\infty$ a.e. for some positive
constants $A_1, A_2$. 
\begin{itemize}
\item[i)] The multiplication operator $	G^{g,g}_0$ is invertible a.e. and, for $\ga_k = (G^{g,g}_0)^{-1} g_k$,
\begin{equation}\label{eq:error}
\norm{I-S_{g,\ga}} \leq \frac{R_{g,g}}{\essinf G^{g,g}_0}\,.
\end{equation}
\item[ii)] If 
\begin{equation}\label{Eq:appdual}
	R_{g,g}< \essinf G^{g,g}_0\,,  
\end{equation}
 then  $\mathcal{G}(\mathbf{g},\mathbf{b})$
and $\mathcal{G}(\boldsymbol\gamma,\mathbf{b})$ are approximately dual frames for $\Lt(\R)$. 
\item[iii)] Assume, additionally, for some  $\delta$-separated set of time-sampling points $\{ a_k\, : \, k\in \Z \}$ and constants $0<p_U,C_L,C_U<\infty$ such that for 
$p_k\in ]2,p_U]\subset\mathbb{R}$, $C_k \in [C_L, C_U]$ we have
\begin{equation}\label{eq:g_k}
\abs{g_k(t)} \leq C_k(1+\abs{t-a_k})^{-p_k}  \mbox{ for all  }k\in\Z\,.
\end{equation}
Then
there exists a sequence $\{b_k^0\}_{k\in\Z}$, such that for all sequence $b_k \leq
b_k^0$, $k\in\Z$,  \eqref{Eq:appdual} holds.
\end{itemize}
\end{proposition}
\begin{remark}
If \eqref{Eq:appdual} holds,
$\mathcal{G}(\boldsymbol\gamma,\mathbf{b})$ is called a single preconditioning dual
system for $\mathcal{G}(\mathbf{g},\mathbf{b})$.
\end{remark}
\begin{proof}
Since all frequency modulation parameters $b_k$ are taken from a closed interval in $\mathbb{R}^+$, the invertibility of $G^{g,g}_0$ is straightforward and  the windows $\ga_k$ are well defined.
Moreover, since $\mathcal{G}(\mathbf{g},\mathbf{b})$ is a Bessel sequence, so
is $\mathcal{G}(\boldsymbol\gamma,\mathbf{b})$, with Bessel bound
$(\essinf G^{g,g}_0)^{-2}B$. Substituting $\ga_k = (G^{g,g}_0)^{-1} g_k$ for
$\ga_k$ in the proof of Lemma~\ref{lem:duality_condition}, the first term in \eqref{eq:lemma} vanishes and we obtain
$(i)$:
\begin{align}\label{eq:spd_duals}
\abs{\inner{f-S_{g,\ga}f}{f}} & \leq \Big |\innerBig{(G^{g,g}_0)^{-1}\sum_{l\in\Z\setminus \{0\}} \sum_{k\in\Z} b_k^{-1} g_k(\cdot)\overline{g_k(\cdot-lb_k^{-1}}) f(\cdot - lb_k^{-1})}{f}\Big |\nonumber \\
& \leq (\essinf G^{g,g}_0)^{-1} \,  \innerBig{\sum_{l\in\Z\setminus \{0\}} \sum_{k\in\Z} b_k^{-1} \abs{g_k(\cdot)}\abs{g_k(\cdot-lb_k^{-1})} \abs{f(\cdot - lb_k^{-1})}}{\abs{f}}\nonumber \\
& \leq (\essinf G^{g,g}_0)^{-1} \, R_{g,g} \norm{f}_2^2 \,.
\end{align} 
 $(ii)$ follows directly from Definition~\ref{Def:AppDual}. Finally,  $(iii)$ follows from  \cite[Theorem~$3.4$]{doma12}, where it is shown that the assumptions \eqref{eq:g_k} on the windows $g_k$
guarantee the existence of  a sequence  $b^0_k$ such that 
\begin{equation}
\inner{S_{g,g} f}{f} \, \norm{f}_2^{-2} \geq 
\essinf G^{g,g}_0 - R_{g,g} > 0\,.
\end{equation}

\end{proof}

Single preconditioning dual windows are a good choice for reconstruction, whenever the frame operator is close to diagonal. This is the case,  if the original windows $g_k$
 decay fast and frequency sampling is fast. \\
 If the frame of interest
is close to some other frame, which, ideally, is better understood,  other  approximate
dual windows may be derived from this frame. The prototypical situation is a NSG frame which is close to a painless NSG frame in the sense of a small perturbation.  Approximately dual frames in the context of perturbation theory were recently studied in \cite{chla10}. In the following proposition
we give error estimates  for the  reconstruction with 
approximately dual frames in such a situation. This provides different reconstruction methods  apart from  single  preconditioning which was addressed in
Proposition~\ref{prop:approx_dual}.

\begin{proposition}\label{prop:approx_dual_perturbation}
Assume that $\mathcal{G}(\mathbf{g},\mathbf{b})$ is a Bessel sequence with bound $B$ and
that $\mathcal{G}(\mathbf{h},\mathbf{b})$ is a NSG frame
with lower and upper frame bound $A_h$ and $B_h$, respectively. We set $\psi_k = h_k - g_k$ and define the following windows:
\begin{align}
(a)\,\ga_{k,l}^1 =& S_{h,h}^{-1}h_{k,l}\,\,  \mbox{ (canonical dual of }h_{k,l}) \\
(b)\,\ga_{k,l}^2 = &S_{h,h}^{-1} g_{k,l}
\end{align}
Then the following hold:
\begin{itemize}
\item[(i)] 
\begin{equation}\norm{I-S_{g,\ga^1}} \leq A_h^{-1/2} \norm{C_\psi}\,.
\label{eq:error_1p}
\end{equation}
If 
$\esssup \sum_{l\in\Z} G^{\psi,\psi}_l < A_h\,, 
$
then
$\mathcal{G}({\boldsymbol\ga}^1,\mathbf{b})$ and
$\mathcal{G}(\mathbf{g},\mathbf{b})$ are approximately dual frames.
\item[(ii)] 
\begin{equation}\label{eq:error_2p}
\norm{I-S_{g,\ga^2}} \leq A_h^{-1} (\sqrt{B_h}+ \sqrt{B}) \norm{C_\psi} .
\end{equation}
If 
$\esssup \sum_{l\in\Z} G^{\psi,\psi}_l <\frac{A_h^2}{(\sqrt{B_h}+ \sqrt{B})^2}\,, 
$
then
$\mathcal{G}({\boldsymbol\ga}^2,\mathbf{b})$ and
$\mathcal{G}(\mathbf{g},\mathbf{b})$ are approximately dual frames.
\end{itemize}
\end{proposition}

\begin{remark}
The first statement of Proposition~\ref{prop:approx_dual_perturbation} is contained in  \cite{chla10}. \\
According to \cite{doma12}, the assumption that $\esssup \sum_{l\in\Z} G^{\psi,\psi}_l <A_h$ can be satisfied if the functions $\psi_k$
decay polynomially, i.e., 
$\abs{\psi_k(t)} \leq C_k (1 + \abs{t})^{-p_k}$
with appropriate constants $C_k$ and  decay rates $p_k>1$ .
\end{remark}

\begin{proof}
From  \eqref{eq:norm_1} it follows that $\norm{C_\psi}^2
\leq \esssup \sum_{l\in\Z}G_l^{\psi,\psi}$. The same estimate holds for
$\norm{U_\psi}^2$. \\
Since $\mathcal{G}(\mathbf{h},\mathbf{b})$ is a frame with canonical dual
frame $\mathcal{G}(\bga^1,\mathbf{b})$, an upper frame bound of $\mathcal{G}(\bga^1,\mathbf{b})$ is given by  $A_h^{-1}$ . We thus obtain (i) as follows:
\begin{equation}
\norm{I - S_{g,\ga^1}} = \norm{U_{\ga^1} C_h -
U_{\ga^1} C_g} \leq \norm{U_{\ga^1}} \norm{C_{\psi}}
\leq A_h^{-1/2} \norm{C_{\psi}}\,.
\end{equation}
If $\esssup \sum_{l\in\Z} G^{\psi,\psi}_l <A_h$, then $\norm{I -
S_{g,\ga^1}} < 1$
 and
$\mathcal{G}(\bga^1,\mathbf{b})$ and
$\mathcal{G}(\mathbf{g},\mathbf{b})$ are approximately dual frames as claimed.

To show $ (ii)$, we 
note that $U_{\ga^2} = S^{-1}_{h,h} U_g$ and thus
\begin{align}
\norm{I-S_{g,\ga^2}} =& \norm{S^{-1}_{h,h}S_{h,h}-S^{-1}_{h,h} S_{g,g}}  = \norm{S^{-1}_{h,h}(U_h C_h -U_g C_g)}\\
\leq&\norm{S^{-1}_{h,h}}\norm{U_h C_h -U_g C_h +U_g C_h-U_g C_g} = A_h^{-1} \norm{U_\psi C_h -U_g C_\psi}\\
\leq & A_h^{-1} \norm{C_\psi} (\sqrt{B_h}+\sqrt{B})
\end{align}and \eqref{eq:error_2p} follows.
\end{proof}

\subsection{Perturbation of painless nonstationary Gabor frames}
If a NSG system can be derived as a perturbation of a painless NSG frame, the approximately dual windows given in Proposition~\ref{prop:approx_dual_perturbation} are
particularly simple to compute. In this situation,  the frame
$\mathcal{G}(\mathbf{h},\mathbf{b})$
is the painless frame $\mathcal{G}(\mathbf{g^o},\mathbf{b})$ and the frame operator $S_{h,h} = S_{g^o,g^o}$
is the multiplication operator $G_0^{g^o,g^o}$. 
Moreover,  in this particular case, the approximately dual frames, given by
$\ga_{k,l}^1 = (G_0^{g^o,g^o})^{-1}g^o_{k,l}$ and
$\ga_{k,l}^2 = (G_0^{g^o,g^o})^{-1} g_{k,l}$ 
are  NSG frames. This is a very important asset, since for NSG frames fast algorithms for analysis and reconstruction using FFT exist. 

In \cite{doma12} we constructed a special class of NSG frames,  arising from painless NSG frames, which we introduce next.
\begin{definition}[Almost painless NSG frames]
Let $\mathcal{G}(\mathbf{g},\mathbf{b})$ be a NSG system, assume that the 
windows $g_k$ are essentially bounded away from zero on the intervals $I_k =
[a_k-(2b_k)^{-1}, a_k+(2b_k)^{-1}]$ and set $g_k^o =
g_k\chi_{I_k}$. If $\mathcal{G}(\mathbf{g^o},\mathbf{b})$ is a (painless) 
frame for $\Lt(\R)$, then we call the system $\mathcal{G}(\mathbf{g},\mathbf{b})$
an {\it almost painless NSG} system (or frame).
\end{definition}

For almost painless NSG systems, the estimates given in Proposition~\ref{prop:approx_dual_perturbation} can be written more explicitly.
\begin{corollary}\label{cor:approx_dual_painless}
Assume that $\mathcal{G}(\mathbf{g},\mathbf{b})$ is an almost painless NSG system and let $A_0 = \essinf G_0^{g^o,g^o}$ to be the lower frame bound of the painless frame $\mathcal{G}(\mathbf{g^o},\mathbf{b})$ and $g_k^r = g_k-g_k\chi_{I_k}$. Then the following hold:
\begin{itemize}
\item[(i)] for $\ga_k^1 = (G_0^{g^o,g^o})^{-1}g^o_k$,  
\begin{equation}\label{eq:error_ga1}
\norm{I-S_{g,\ga^1}} \leq A_0^{-1} \sqrt{R_{g^r,g^o}\cdot R_{g^o,g^r}}
\end{equation}
\item[(ii)] for $\ga_k^2 = (G_0^{g^o,g^o})^{-1} g_k$
\begin{equation}\label{eq:error_ga2}
\norm{I-S_{g,\ga^2}} \leq A_0^{-1} \left (R_{g^r,g^o} + R_{g^o,g^r} + \esssup \sum_{l\in\Z} G_l^{g^r,g^r} \right )
\end{equation}
\end{itemize}

\end{corollary}

\begin{proof}
The estimates follow from Lemma~\ref{lem:duality_condition}. First, substituting $\ga_k^1 = (G^{g^o,g^o}_0)^{-1} g^o_k$ for
$\ga_k$ in \eqref{eq:duality_estimate}, the first term vanishes, since $\sum_{k\in\Z} b_k^{-1} \overline{g_k} g_k^o = G^{g^o,g^o }_0$, and we obtain 
\begin{equation*}
\norm{I-S_{g,\ga^1}} \leq \sqrt{R_{g,\ga^1}\cdot R_{\ga^1,g}} \leq (\essinf G^{g^o,g^o}_0)^{-1} \,  \sqrt{R_{g^r,g^o}\cdot R_{g^o,g^r}}\,,
\end{equation*}
since
\begin{align}\label{eq:Rrr}
R_{g,\ga^1} &= \esssup \sum_{l\in\Z\setminus \{ 0\}} \sum_{k\in\Z} b_k^{-1} \abs{g_k(\cdot - lb_k^{-1})} \abs{(G^{g^o,g^o}_0)^{-1} g^o_k(\cdot)} \\\nonumber 
&\leq (\essinf G^{g^o,g^o}_0)^{-1} \cdot \esssup \sum_{l\in\Z\setminus \{ 0\}} \sum_{k\in\Z} b_k^{-1} \abs{g^r_k(\cdot - lb_k^{-1})} \abs{g^o_k(\cdot)} \\\nonumber  
&= (\essinf G^{g^o,g^o}_0)^{-1} \, R_{g^r,g^o}\,,
\end{align}
similarly for $R_{\ga^1,g}$.

For $(ii)$ we substitute $\ga_k^2 = (G_0^{g^o,g^o})^{-1} g_k$ for $\ga$ in the proof of Lemma~\ref{lem:duality_condition}, $g^o_k +g^r_k$ for $g_k$ and use the fact that $g_k^o$ and $g_k^r$ have disjoint supports. Then, since $\overline{g^o_k(\cdot - lb_k^{-1})}g^o_k(\cdot)$ is zero for $l\neq 0$, using \eqref{eq:estimate_R} we obtain that
\begin{align*}
\abs{\inner{f-S_{g,\ga^2}f}{f}} &= \absBig{\innerBig{f - \sum_{k\in\Z} \sum_{l\in\Z} \overline{g_k(\cdot - lb_k^{-1})} (G^{g^o,g^o}_0)^{-1} g_k(\cdot ) f(\cdot - lb_k^{-1})}{f}} \\ \nonumber
&\leq \absBig{\innerBig{f - (G^{g^o,g^o}_0)^{-1} \sum_{k\in\Z} b_k^{-1} \abs{g_k^o}^2 f}{f}} \\\nonumber 
&+ (\essinf (G^{g^o,g^o}_0)^{-1} \left ( R_{g^r,g^o} + R_{g^o,g^r} + \norm{S_{g^r,g^r}} \right ) \norm{f}_2^2\,.
\end{align*}
The first term vanishes and by \eqref{eq:norm_1}, $\norm{S_{g^r,g^r}} \leq \esssup \sum_{l\in\Z} G_l^{g^r,g^r}$.
\end{proof}

\section{Examples}

We present  two examples to illustrate our theory. In the first example we  deal with
almost painless frames using Gaussian windows. We consider three different approximately dual frames and check their performance in terms of reconstruction.

In both examples,  we  consider a basic window and dilations by $2$ and
$\frac{1}{2}$, respectively. Since the dilation parameters take only three
different values, there are three kinds of windows,  with support size $1/2$,
$1$ and $2$, respectively. Note that, while theoretically possible, sudden
changes in the shape and width of adjacent windows turn out to be undesirable
for applications, hence we only allow for stepwise change in dilation
parameters.
	
\begin{example}\label{Ex1}
Let $s_k \in \{-1,0,1\}$ with $\abs{s_k-s_{k-1}} \in \{0,1\}$ for all $k\in\Z$.
We consider a sequence of windows $g_k$ that are translated and dilated versions
of the Gaussian window $g(t) = e^{-\pi[\sigma t]^2}$: $g_k(t) =
T_{a_k}\sqrt{b_k}g(b_k t) = \sqrt{b_k}g(b_k (t-a_k))$, with $\sigma=2.5$,  $b_k=2^{s_k}$, $a_0=0$
and for all $k\in\Z$ 
\begin{align*}
a_{k+1} &= a_k + (2b_k)^{-1} \,\quad \,\,\mbox{if} \quad s_k=s_{k+1}\,,\\
a_{k+1} &= a_k + (3b_{k+1})^{-1} \quad \mbox{if} \quad
s_k>s_{k+1}\,,\\
a_{k+1} &= a_k + (3b_k)^{-1} \,\,\, \quad \,\mbox{if} \quad
\,\, s_k<s_{k+1}\,.
\end{align*}
Here, $b_L = 1/2$,  $b_U =2$ and the $\{a_k\, : \, k\in\Z\}$ are separated with minimum distance  $\delta =
1/4$. 
We arrange the windows as follows: after each change of window size, no change
is allowed in the next step; in other words, each window has at least one
neighbor of the same size. 

Let $I_k=[a_k-(2b_k)^{-1},a_k + (2b_k)^{-1}]$ and define a new set of windows
by $g^o_k(t) = g_k(t)\chi_{I_k}$. Then $\{M_{lb_k} g^o_k\, : \, k,l\in\Z\}$ is a
painless nonstationary Gabor frame with lower frame bound $A_0 = 0.1718$.

The system $\mathcal{G}(\mathbf{g},\mathbf{b})$ arises from the painless
frame $\mathcal{G}(\mathbf{g^o},\mathbf{b})$, and therefore we are interested in the approximate dual windows proposed in 
Corollary~\ref{cor:approx_dual_painless}. We first consider
$\ga_k^1 = (G_0^{g^o,g^o})^{-1} g_k^o$, the canonical dual frame of
$\mathcal{G}(\mathbf{g^o},\mathbf{b})$. According to \eqref{eq:error_ga1},  we 
need to calculate $R_{g^r,g^o}$ and $R_{g^o,g^r}$ in order to obtain an  estimate of the reconstruction error.\\
\emph{Claim~1:} $R_{g^r,g^o} \leq 0.00827 + \mathcal{O}(10^{-6})$.\\
For fixed $k\in\Z$, 
\begin{equation*}
b_k^{-1}\abs{g_k^o(t)}\sum_{l\in\Z\setminus\{0\}} \abs{g^r_k(t -
lb_k^{-1})} = b_k^{-1/2}\abs{g_k^o(t)}\sum_{l\in\Z} b_k^{-1/2}\abs{g^r_k(t -
lb_k^{-1})}
\end{equation*}
since $g^r_k$ and $g_k^o$ have disjoint support. 
Then, due to  $b_k^{-1}$-periodicity of $\sum_{l\in\Z} b_k^{-1/2}\abs{g^r_k(t -
lb_k^{-1})}$, we have
\begin{align}
R_{g^r,g^o} &= \esssup \sum_{k\in\Z} b_k^{-1/2} \abs{g_k^o(\cdot)}
\sum_{l\in\Z} b_k^{-1/2} \abs{g_k^r(\cdot - lb_k^{-1})} \nonumber \\
&\leq \esssup \sum_{k\in\Z} b_k^{-1/2} \abs{g_k^o(\cdot)} \cdot 
\esssup_{t\in\I_k} \sum_{l\in\Z} b_k^{-1/2}\abs{g^r_k(t -lb_k^{-1})}\label{Eq:Rgg_est}
\end{align}
In order to obtain  more accurate estimates, we split $I_k$ by setting $I_k^+ = [a_k,a_k+(2b_k)^{-1}]$ and $I_k^- = [a_k-(2b_k)^{-1},a_k]$ and estimate  expression \eqref{Eq:Rgg_est} by  its values  at the end points of $I_k^+$ or  $I_k^-$, respectively. We observe that, for $t\in I_k^+$:
\begin{align*}
\sum_{l\in\Z} b_k^{-1/2}\abs{g^r_k(t - lb_k^{-1})} &=
\sum_{l=1}^{\infty} \Big [ e^{-\pi \left [\sigma
b_k(t-a_k-lb_k^{-1})\right ]^2} +
e^{-\pi \left [\sigma b_k(t-a_k+lb_k^{-1}) \right ]^2} \Big ] \\ \nonumber 
& \leq  \sum_{l=1}^{\infty} e^{-\pi \left
[\sigma l\right ]^2} +\sum_{l=1}^{\infty}
e^{-\pi \left [\sigma (2l-1)/2 \right ]^2} \leq 0.00738 + \mathcal{O}(10^{-6})\,.
\end{align*}
By the symmetry of $g_k^r$ with respect to $a_k$, an analogous estimate holds for $t\in I_k^-$. Furthermore, due to the arrangement of the windows
$g_k$, $ \esssup \sum_{k\in\Z} b_k^{-1/2} \abs{g_k^o(\cdot)} \leq 1.1206$ and \emph{Claim~1} follows.\\
\emph{Claim~2:} $R_{g^o,g^r}\leq 0.0157 + \mathcal{O}(10^{-6})$.\\

\begin{figure}[h]
\centerline{\includegraphics[width=\textwidth]{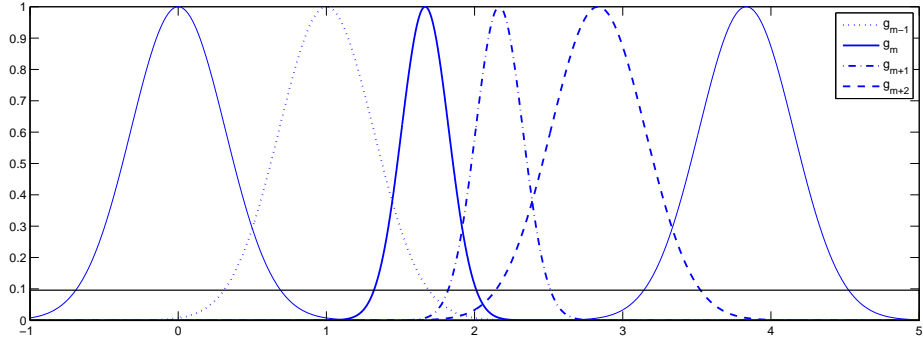}}
 	\caption{An example of an arrangement of windows $g_k$. }
\label{fig:windows}
\end{figure}

\begin{figure}[h]
\centerline{\includegraphics[width=\textwidth]{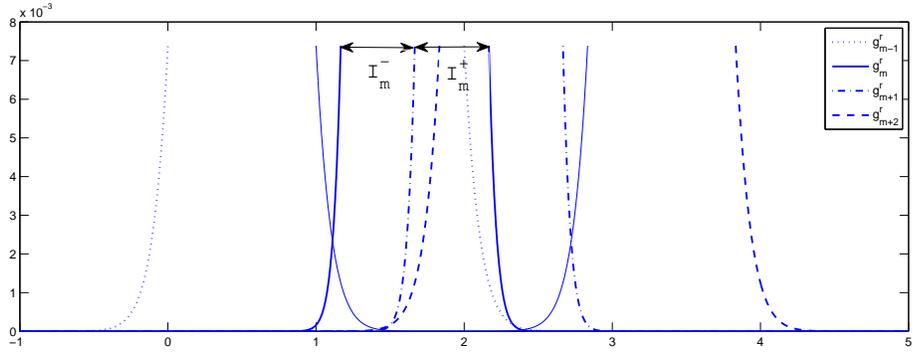}}
 	\caption{A zoom-in on the tails $g_k^r$ in the arrangement of Fig.~\ref{fig:windows}. }
	\label{fig:tails}
\end{figure}

Observe that $\sum_{l\in\Z\setminus \{ 0\}}
b_k^{-1/2}\abs{g_k^o(t - lb_k^{-1})}\leq \norm{b_k^{-1/2}g^o_k}_{\infty}=1$.
Then, with $I_m^+$ and $I_m^-$  defined as before:
\begin{align*}
R_{g^o,g^r} &= \esssup \sum_{k\in\Z} b_k^{-1/2} \abs{g_k^r(\cdot)}
\sum_{l\in\Z\setminus \{ 0\}} b_k^{-1/2} \abs{g_k^o(\cdot - lb_k^{-1})} \leq 
\esssup \sum_{k\in\Z} b_k^{-1/2} \abs{g_k^r(\cdot)} \\ \nonumber
& \leq \sup_{m\in\Z} \Big \{\esssup_{t\in I_{m}^+}
\sum_{k\in\Z} b_k^{-1/2} \abs{g_k^r(t)}\, , \, \esssup_{t\in I_{m}^-} \sum_{k\in\Z}
b_k^{-1/2} \abs{g_k^r(t)} \Big \} \\ \nonumber
& = \sup_{m\in\Z} \Big \{ \underbrace{\esssup_{t\in I_{m}^+}
\sum_{k\in\Z} e^{-\pi [\sigma b_k(t-a_k)]^2} \chi_{I_k^c}(t)}_{C^+}\,,\,
\underbrace{\esssup_{t\in I_{m}^-} \sum_{k\in\Z} e^{-\pi [\sigma
b_k(t-a_k)]^2} \chi_{I_k^c}(t)}_{C^-} \Big \}.
\end{align*}
We bound  $C^+$ and $C^-$ by their maximal values on $I_m^+$, respectively $I_m^-$. The set $\{a_k\, : \, k\in\Z\}$ is $\delta-$separated, 
hence $\abs{a_k - a_m} \geq \abs{k-m} \delta$. By the arrangement of the
windows, there are at most two windows $g_k^r$ 
which assume their  maximum  $e^{-\pi [\sigma/2 ]^2}$  in the interval $I_m^+$ or in $I_m^-$.
 Without loss of generality we assume that the two maximal values occur in  $I_m^+$,   corresponding to the  windows $g_{m-1}^r$ and $g_{m+2}^r$.
 Then,  $g_{m+1}^r$, $g_m^r$ are zero in $I_m^+$, cf.~Figure~\ref{fig:tails} for an example situation. Therefore,
\begin{align*}
C^+ &\leq \sum_{k > m + 2 } e^{-\pi [\sigma b_k(a_m + (2b_m)^{-1} -
a_k)]^2}  + 2e^{-\pi [\sigma b_k (2b_k)^{-1}]^2} + \sum_{k < m -1} e^{-\pi
[\sigma b_k(a_m - a_k)]^2} \\\nonumber
& \leq 2 \left ( e^{-\pi [\sigma/2 ]^2} +  \sum_{k>0;\, kb_k > 2}e^{-\pi [\sigma
b_k k \delta]^2} \right )\,.
\end{align*}
Since we assumed that two windows $g_k^r$ reached their maximum in $I_m^+$, $C^- \leq C^+$. As before, the bound from Claim~2 follows by  numerical calculations.\\

In summary, due to  \eqref{eq:error_ga1}, Claim~1 and~2 and the lower frame bound $A_0 = 0.1718$, the reconstruction error using the approximate duals $\ga_k^1$ is bounded  by
\begin{equation}
\norm{I - S_{g,\ga^1}} \leq 0.0663 + \mathcal{O}(10^{-6})\,.
\end{equation}

Another choice of approximate dual system are the windows $\ga_k^2 = (G_0^{g^o,g^o})^{-1} g_k$. 
In this setting we use \eqref{eq:error_ga2}  to derive 
\begin{equation*}
\norm{I - S_{g,\ga^2}} \leq A_0^{-1} \left (R_{g^r,g^o} + R_{g^o,g^r} + \esssup
\sum_{l\in\Z} G_l^{g^r,g^r} \right ) \leq 0.1402 + \mathcal{O}(10^{-6})\,,
\end{equation*}
since, by  previous calculations, 
\begin{align}\label{eq:estimate_Ggrgr}
\esssup \sum_{l \in\Z} G_l^{g^r,g^r} &= \esssup \sum_{k\in\Z} b_k^{-1/2}
\abs{g_k^r(t)} \sum_{l\in \Z} b_k^{-1/2}\abs{g_k^r(t-lb_k^{-1})} \\\nonumber 
&\leq \esssup \sum_{k\in\Z} b_k^{-1/2} \abs{g_k^r(t)} \, \cdot \, 
\esssup \sum_{l\in \Z} b_k^{-1/2}\abs{g_k^r(t-lb_k^{-1})} \\\nonumber
&\leq 0.0001158 + \mathcal{O}(10^{-6})\,.
\end{align}

As a third choice of approximate dual windows we consider single preconditioning windows 
$\ga_k=(G_0^{g,g})^{-1}g_k$ introduced in Proposition~\ref{prop:approx_dual}. It can 
easily be seen that
\begin{equation}
R_{g,g}  \leq R_{g^o,g^r} + R_{g^r,g^o} + R_{g^r,g^r} \leq  R_{g^o,g^r} +
R_{g^r,g^o} + \esssup \sum_{l\in\Z} G_l^{g^r,g^r}\,,
\end{equation}
and, by previous calculations, it follows
\begin{equation}
R_{g,g} \leq 0.0241 + \mathcal{O}(10^{-6}) < A_0 \leq \essinf G_0^{g,g}\,.
\end{equation}
Therefore, by Proposition~\ref{prop:approx_dual} (ii),
$\G(\mathbf{g},\mathbf{b})$ and $\G(\mathbf{\ga},\mathbf{b})$ are approximately
dual frames. Moreover
\begin{equation}
\norm{I-S_{g,\ga}} \leq  \frac{R_{g,g}}{\essinf G_0^{g,g}} \leq
\frac{R_{g,g}}{A_0} \leq 0.1402 + \mathcal{O}(10^{-6})\,.  
\end{equation}
\end{example}
\begin{remark}
Observe that from each of the approximately dual frame estimates given in the above example, the frame property of  $\G(\mathbf{g},\mathbf{b})$ follows.\\
Note that, from \eqref{eq:estimate_Ggrgr}, the frame property of
$\G(\mathbf{g},\mathbf{b})$ may also be derived by applying results from perturbation theory~cf.~\cite{Ch98}. Indeed, since $\sum_{k,l\in\Z} \abs{\inner{f}{g_{k,l} - g^o_{k,l}}}^2 =
\sum_{k,l\in\Z} \abs{\inner{f}{g^r_{k,l}}}^2 \leq \esssup \sum_{l \in\Z}
G_l^{g^r,g^r} \norm{f}_2^2 < A_0 \norm{f}_2^2$, it
follows that $\G(\mathbf{g},\mathbf{b})$ is a frame with a lower frame
bound $A = 0.1630$.
\end{remark}
\begin{example}\label{Ex2}
In our second example, we turn to the situation mentioned in the introduction, namely, the construction of non-uniform filter banks with compactly supported, that is, FIR filters, via NSG frames. In this situation, the frame operator $S$ does not have a Walnut-like structure as given in \eqref{eq:walnut} on the time side. However, $S$  may be considered on the frequency side by applying a Fourier transform. Then, we encounter the same structure as before and may exploit the developed techniques to deduce the frame property and to construct approximate dual frames for reconstruction. The situation is schematically depicted in Figure~\ref{Fig:Win_schem}. 

\begin{figure}[h]
\centerline{\includegraphics[width=14cm,height=8cm]{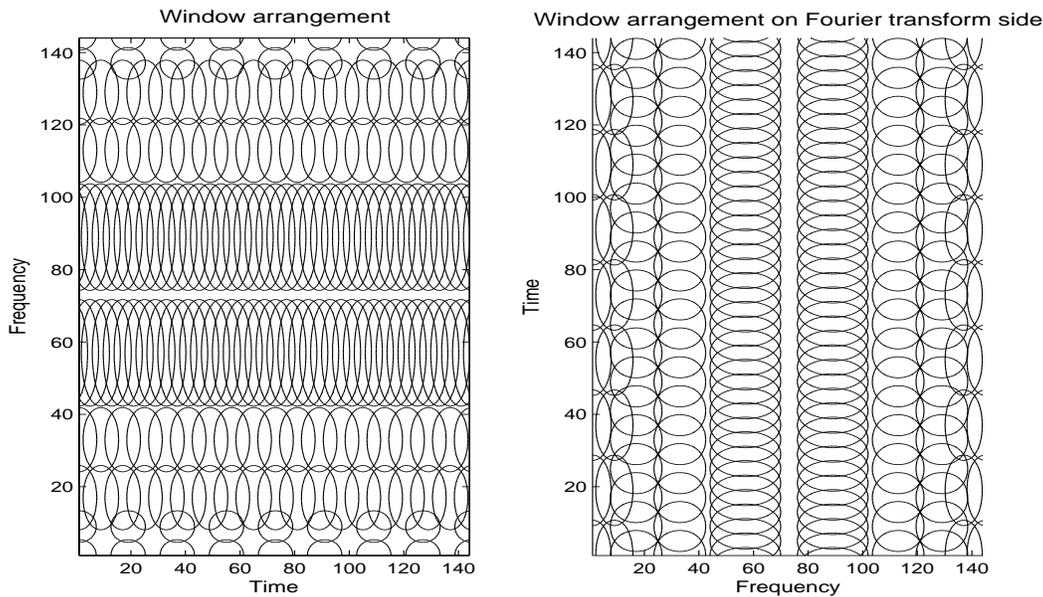}}
 	\caption{An example for the arrangement of dilated windows in Example 2}
	\label{Fig:Win_schem}. 
\end{figure}
It is obvious that, while, in the time domain, various windows $g_k$ with different pass-bands, are shifted to obtain the overall system, applying a Fourier transform yields the known situation: various windows $\hat{g}_k$ are modulated to create an NSG system. \\
On the other hand, we are interested in using FIR filters, that is, the windows $g_k$ are all compactly supported, hence not bandlimited. Similar to the construction in Example~\ref{Ex1}, we can cut the windows 
$\hat{g}_k$ to obtain a painless NSG reference frame.\\
More precisely, we consider the family of windows $g_k$ and a vector of corresponding time-shift parameters $a_k$. Then, we set $g_{k,l}  = T_{a_k l} g_k$ and are interested in the frame property of the set of functions $\{g_{k,l}:\; k,l\in\mathbb{Z}\}$. Considering, due to the lack of structure on the time side as mentioned above,  the corresponding frame operator   on the frequency side corresponds to investigating the operator $\mathcal{F} S \mathcal{F}^\ast$, which acts on the Fourier transform of a signal of interest. In other words, we are now dealing with the set of functions $\{\mathcal{F}(g_{k,l}) = M_{a_k l} \hat{g}_k:\; k,l\in\mathbb{Z}\}$. \\
For the current example, we consider Hanning windows $h_k$ and, as in the previous example, apply dilations by $2^{-1}$ and $2$, respectively, to obtain various time- and frequency resolutions. The time-shift parameters $a_k$ are chosen in parallel to the choice of the frequency-shift parameters $b_k$ in Example~\ref{Ex1}.  
 
Given the explicit knowledge of the spectral properties of the Hanning windows, explicit error estimates can be derived in a similar manner as in the previous example. 
Here, we also numerically calculate the errors resulting from reconstruction by means of  the three different proposed approximate dual systems. We use the same nomenclature as before, that is, $\gamma^1_k$, $\gamma^2_k$ and $\gamma^3_k$ denote the canonical duals of the painless frame, the set of windows $(G_0^{g^o,g^o})^{-1} g_k$ and the single preconditioning windows, respectively. Then
\begin{enumerate}
\item $\norm{I - S_{g,\ga^1}} = 0.0210$
\item $\norm{I - S_{g,\ga^2}} = 0.0407$
\item $\norm{I - S_{g,\ga^3}} =  0.0407$
\end{enumerate}
As before, the canonical duals of the painless frame provide the best approximate reconstruction.\\
The set of windows $\hat{g}_k$ used in this example, together with their true, single preconditioning duals $\gamma^3_k$ and the canonical duals of the corresponding painless frame, $\gamma^1_k$, are depicted in Figure~\ref{Fig:Dual1}. On the right plots, zoom-ins are shown to better compare the detailed behavior. \\
\begin{figure}[h]
\centerline{\includegraphics[width=14cm,height=8.2in]{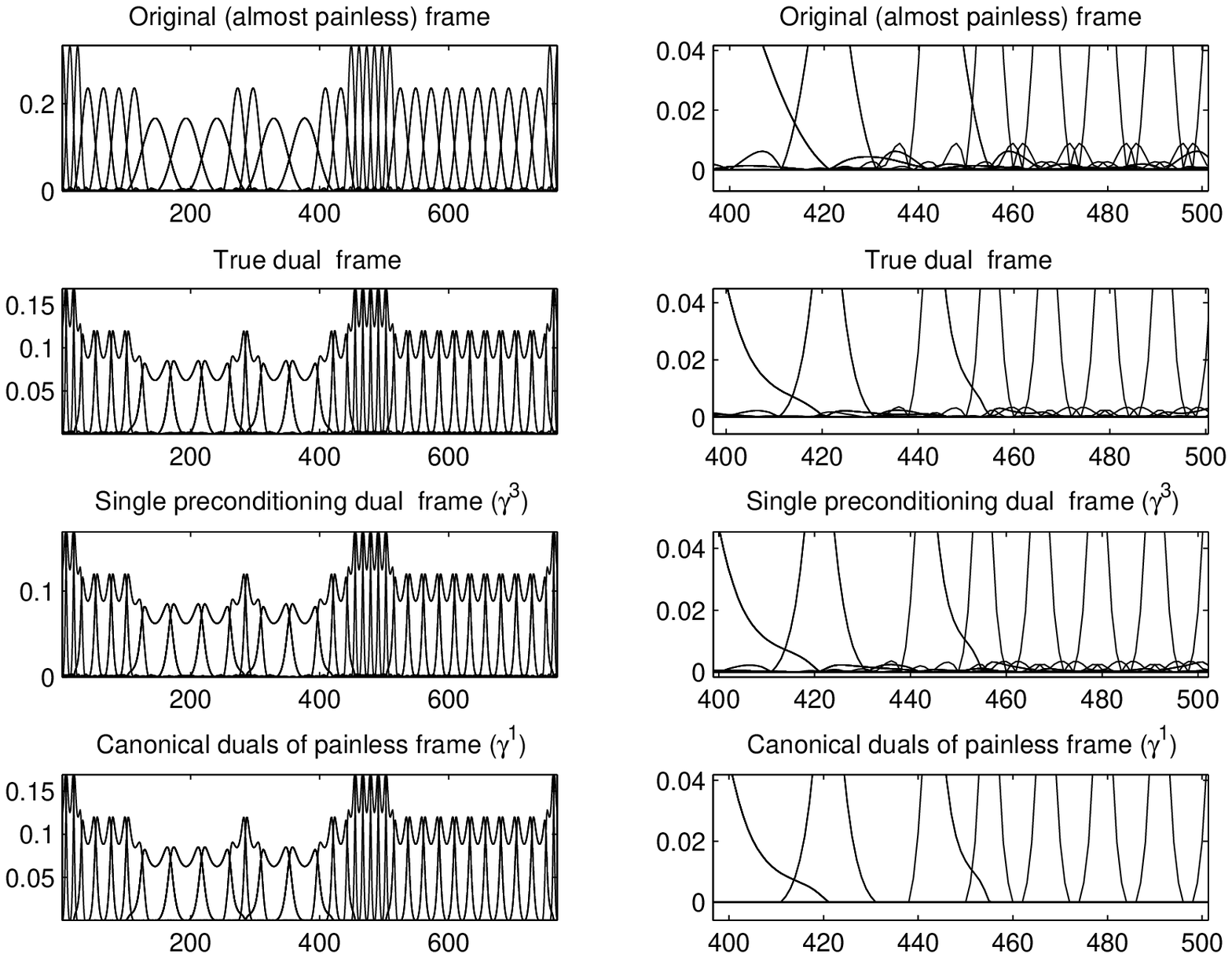}}
 	\caption{The original windows and the windows from  various (approximate) dual frames.}
	\label{Fig:Dual1}. 
\end{figure}
A comparison between a basic window from the true dual frame and the preconditioning duals $\gamma^3_k$ and  the true dual and the approximate dual  $\gamma^1_k$  is depicted in Figures~\ref{Fig:True_pre} and~\ref{Fig:True_can}, respectively, in both the frequency domain (upper plot) and the time domain (lower plot). It should be noted that, while compactly supported in frequency, $\gamma^1_k$  still have better decay in the time domain than both the true dual windows and the other approximate dual windows.
\begin{figure}[hbt]
	\begin{center}
	\includegraphics[width=14cm,height=8cm]{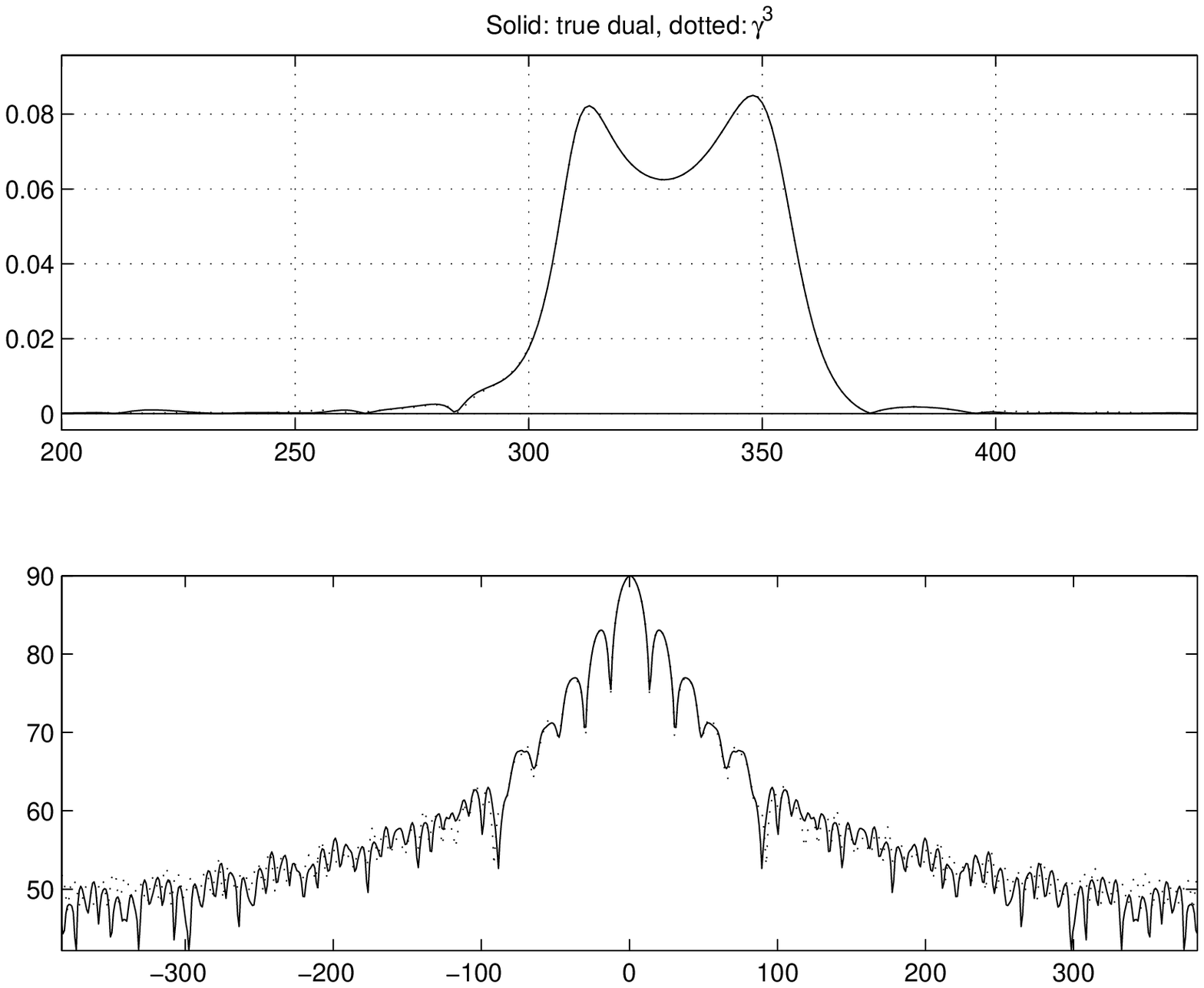}
	\end{center}\caption{Comparison between true dual and single preconditioning dual $\gamma^3$.}	\label{Fig:True_pre}
\end{figure}
 
 \begin{figure}[hbt]
	\begin{center}
		\includegraphics[width=14cm,height=8cm]{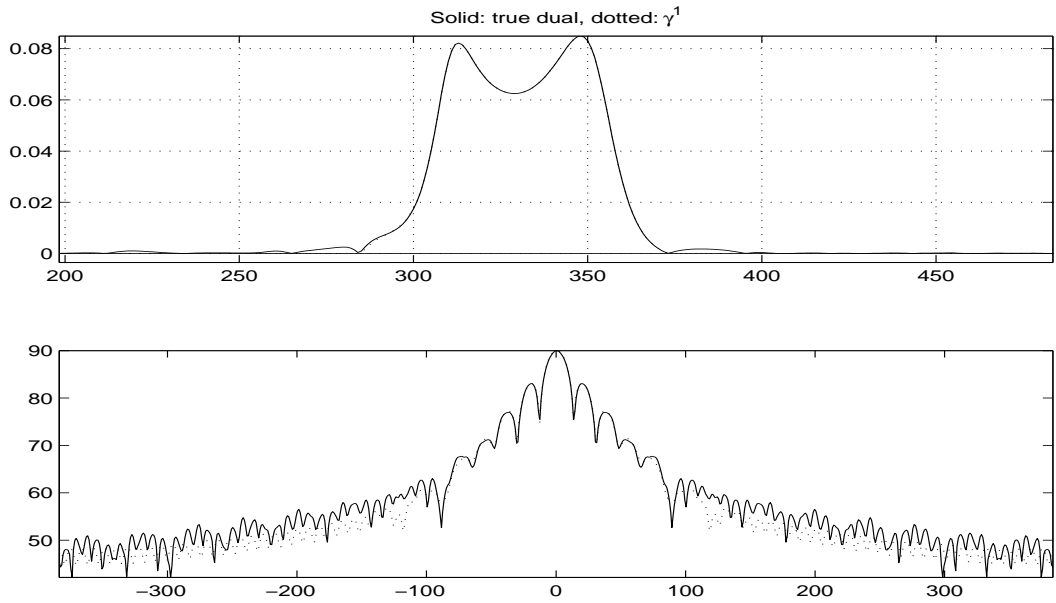}
	\end{center}\caption{Comparison between true dual and painless approximate  dual $\gamma^1$.}	\label{Fig:True_can}
\end{figure}
\end{example}



\end{document}